\documentclass[11pt,reqno]{amsart}

\usepackage{amsmath,amsthm,amssymb,verbatim}
\usepackage[tmargin=1.0in,bmargin=1.0in,rmargin=1.0in,lmargin=1.0in]{geometry}
\usepackage[breaklinks=true]{hyperref}

\theoremstyle{plain}

\newtheorem{theorem}{\textrm{\textbf{Theorem}}}[section]

\newtheorem{corollary}[theorem]{\textrm{\textbf{Corollary}}}
\newtheorem{proposition}[theorem]{\textrm{\textbf{Proposition}}}
\newtheorem{lemma}[theorem]{\textrm{\textbf{Lemma}}}

\theoremstyle{definition}

\newtheorem{definition}[theorem]{\textrm{\textbf{Definition}}}

\newtheorem{remark}[theorem]{\textrm{\textbf{Remark}}}

\theoremstyle{remark}

\numberwithin{equation}{section}

\def\R{\mathbb{R}}

\def\N{\mathbb{N}}
\def\calh{\mathcal H}
\newcommand{\bp}{\ensuremath{\mathbb P}}
\newcommand{\calm}{\ensuremath{\mathcal M}}

\DeclareMathOperator{\Id}{\ensuremath{Id}}

\begin{document}

\title{Critical exponents of graphs}

\author{Dominique Guillot \and Apoorva Khare \and Bala
Rajaratnam}

\thanks{D.G., A.K., and B.R.~are partially supported by the following: US
Air Force Office of Scientific Research grant award FA9550-13-1-0043, US
National Science Foundation under grant DMS-0906392, DMS-CMG 1025465,
AGS-1003823, DMS-1106642, DMS-CAREER-1352656, Defense Advanced Research
Projects Agency DARPA YFA N66001-11-1-4131, the UPS Foundation,
SMC-DBNKY, and an NSERC postdoctoral fellowship}

\address[D.G.]{Department of Mathematical Sciences, University of
Delaware, Newark, DE - 19716, USA}

\address[A.K. and B.R.]{Departments of Mathematics and Statistics,
Stanford University, Stanford, CA - 94305, USA}

\email{D.G. dguillot@udel.edu; A.K. khare@stanford.edu; B.R.
brajaratnam01@gmail.com}

\date{\today}


\begin{abstract}
The study of entrywise powers of matrices was originated by Loewner in
the pursuit of the Bieberbach conjecture. Since the work of FitzGerald
and Horn (1977), it is known that $A^{\circ \alpha} := (a_{ij}^\alpha)$
is positive semidefinite for every entrywise nonnegative $n \times n$
positive semidefinite matrix $A = (a_{ij})$ if and only if $\alpha$ is a
positive integer or $\alpha \geq n-2$. This surprising result naturally
extends the Schur product theorem, and demonstrates the existence of a
sharp phase transition in preserving positivity. In this paper, we study
when entrywise powers preserve positivity for matrices with structure of
zeros encoded by graphs.  To each graph is associated an invariant called
its \textit{critical exponent}, beyond which every power preserves
positivity.
In our main result, we determine the critical exponents of all
chordal/decomposable graphs, and relate them to the geometry of the
underlying graphs. We then examine the critical exponent of important
families of non-chordal graphs such as cycles and bipartite graphs.
Surprisingly, large families of dense graphs have small critical
exponents that do not depend on the number of vertices of the graphs.
\end{abstract}
\maketitle

\section{Introduction and main result}\label{Sintro}

Let $f: \R \to \R$ be a real function and denote by $\bp_n$ the cone of
$n \times n$ real symmetric positive semidefinite matrices. The function
$f$ naturally operates entrywise on $\bp_n$ by defining $f[A] :=
(f(a_{ij}))$. Whether or not the mapping $A \mapsto f[A]$ preserves
positivity (i.e., $f[A] \in \bp_n$ for all $A \in \bp_n$) is an important
problem that has been well-studied in the literature - see e.g.
Schoenberg \cite{Schoenberg42}, Rudin \cite{Rudin59}, Herz \cite{Herz63},
Horn \cite{Horn}, Christensen and Ressel \cite{Christensen_et_al78},
Vasudeva \cite{vasudeva79}, and FitzGerald et al \cite{fitzgerald}. In
one of their main results in the area, Schoenberg and Rudin
(\cite{Schoenberg42}, \cite{Rudin59}) have shown an important
characterization of functions $f$ that preserve positivity on $\bp_n$ for
\textit{all} $n$. Their results show that such functions are precisely
the absolutely monotonic functions (i.e., they are analytic with
nonnegative Taylor coefficients).

When the dimension $n$ is fixed, obtaining useful characterizations of
functions preserving positivity is difficult, and very few results are
known. In the pursuit of the Bieberbach conjecture (de Branges's
theorem), Loewner was led to study which real powers $\alpha > 0$
preserve positivity (i.e., positive semidefiniteness) for all $n \times
n$ positive semidefinite matrices with positive entries. As a consequence
of the Schur product theorem, every integer power trivially preserves
positivity when applied entrywise. Identifying the other real powers that
do so is a much more challenging task. The problem was solved by
FitzGerald and Horn in 1977; in their landmark paper \cite{FitzHorn},
they show that a real power $\alpha > 0$ preserves positivity on $n
\times n$ matrices if and only if $\alpha$ is a positive integer or
$\alpha \geq n-2$. The work of FitzGerald and Horn was later extended in
different directions by multiple authors including Bhatia and Elsner
\cite{Bhatia-Elsner}, Hiai \cite{Hiai2009}, and Guillot, Khare and
Rajaratnam \cite{GKR-crit-2sided} - see \cite{GKR-crit-2sided} for a
history of the problem and recent developments.

In this paper, we significantly generalize the original problem by
studying which entrywise powers preserve positivity when the matrices
have an additional structure of zeros that is encoded by a graph.
Motivation for this problem comes from its connection to the
regularization of covariance/correlation matrices in high-dimensional
probability and statistics. The study of entrywise functions preserving
positivity has recently received renewed attention in this area. For
instance, powering up is a way to effectively and efficiently separate
out signal from noise; see e.g.~\cite{Li_Horvath, Zhang_Horvath}. More
generally, it is common to use entrywise functions in high-dimensional
probability and statistics to regularize covariance/correlation matrices
and improve their properties (e.g., condition number, Markov random field
structure, etc.) - see \cite{bickel_levina, Guillot_Rajaratnam2012,
Guillot_Rajaratnam2012b, hero_rajaratnam, Hero_Rajaratnam2012}.
Preserving positivity is critical for such techniques to be useful in
downstream applications. Moreover, in that setting, preserving positivity
for all dimensions as in the classical setting of Schoenberg's result is
unnecessarily restrictive since the dimension of the problem is given. It
is thus natural to ask which powers preserve positivity for matrices of a
given size.

In recent work by the authors \cite{GKR-lowrank,
Guillot_Khare_Rajaratnam2012, Guillot_Rajaratnam2012b}, classical results
by Schoenberg and Rudin were extended in fixed dimensions to various
settings motivated by modern-day applications.
These include: 1) characterizing functions preserving positivity when
applied only to the off-diagonal elements of matrices (as is often the
case in applications), 2) preserving positivity under rank constraints,
and 3) preserving positivity under sparsity constraints. Motivation for
the second problem comes from the fact that the rank of
covariance/correlation matrices corresponds to the sample size used to
estimate them, and that such matrices are very often estimated from small
samples in modern-day applications. The third problem is motivated by the
fact that sparsity in covariance/correlation matrices generally
corresponds to independence or conditional independence of the
corresponding random variables. The problem of regularizing matrices with
an original sparsity structure thus naturally occurs when there is prior
knowledge available about these dependencies.  

The present paper focuses on matrices with prescribed structure of zeros.
Such matrices naturally occur in combinatorics (e.g.~ in spectral graph
theory) and many other areas of mathematics - see
e.g.~\cite{Agler_et_al_88, Bala-PD-completions, Fallat-Hogben,
Grone-PD-completions}. These matrices also occur naturally in multiple
fields of the broader mathematical sciences such as optimization, network
theory, and in modern high-dimensional probability and statistics. 

Let $\N$ denote the set of positive integers.
Given $n \in \N$ and $I \subset \R$, let $\bp_n(I)$ denote the set of
symmetric positive semidefinite $n \times n$ matrices with entries in
$I$. Given a simple graph $G = (V,E)$ with nodes $V = \{ 1,2,\dots,n\}$,
and a subset $I \subset \R$, define
\begin{equation}
\bp_G(I) := \{ A \in \bp_n(I) : a_{ij} = 0\ \forall (i,j) \not\in E,\
i\ne j \}.
\end{equation}
For simplicity, we denote $\bp_G(\R)$ by $\bp_G$. All graphs in the
remainder of the paper are assumed to be finite and simple.

An important family of graphs in mathematics as well as in applications
is the family of \textit{chordal graphs} (see e.g.~\cite[Chapter
5.5]{Diestel}, \cite[Chapter 4]{golumbic}). Recall that chordal graphs
(also known as decomposable graphs, triangulated graphs, or rigid circuit
graphs) are graphs in which all cycles of four or more vertices have a
chord. Chordal graphs are perfect, and have a rich structure that has
been well-studied in the literature. They also play a fundamental role in
multiple areas including the matrix completion problem (see
e.g.~\cite{Bala-PD-completions,Grone-PD-completions,paulsen_et_al}),
maximum likelihood estimation in the theory of Markov random fields
\cite[Section 5.3]{lauritzen}, and perfect Gaussian elimination
\cite{golumbic}. For example, when solving sparse linear systems, it is
important to preserve the structure of zeros of the original matrix for
storage and computation efficiency purposes. By a result of Golumbic
\cite[Theorem 12.1]{golumbic}, Gaussian elimination can be performed on a
given sparse matrix without ever changing a zero entry to a nonzero entry
if and only if  the structure of zeros of the matrix forms a chordal
graph. 

In our main result, we characterize all the powers $\alpha \geq 0$ that
preserve positivity on $\bp_G([0,\infty))$ for each chordal graph $G$. 
We also characterize all $\alpha \in \R$ for which the odd and even
extensions to $\R$ of the power functions preserve positivity on
$\bp_G(\R)$. Given FitzGerald and Horn's result mentioned above, it is
natural to believe that the critical exponent of a graph $G$ is
$\omega(G)-2$, where $\omega(G)$ is the clique number of $G$.
Surprisingly, for non-complete graphs, this is not always the case.
Nevertheless, as we demonstrate in the paper, the set of powers
preserving positivity on $\bp_G$  for a chordal graph $G$ is driven by
the existence of specific subgraphs of $G$. Our results thus connect the
discrete structure of the graph $G$ to the analytic properties of the
cone of positive semidefinite matrices $\bp_G$, and naturally extend the
complete graph case that was studied by FitzGerald and Horn in
\cite{FitzHorn}.

Imposing the additional constraints on the structure of zeros leads to
challenging problems. For example, in the case of chordal graphs,
characterizing powers preserving positivity requires intricate arguments
involving Loewner super-additive functions. More generally, many
familiar constructions that can be used to study positivity (e.g.,
working with Schur complements) generally fail to preserve the underlying
structure of zeros. As a consequence, many new techniques have to be
developed to address such issues.

In the last section of the paper, we also determine the set of powers
preserving positivity for many broad families of non-chordal graphs
including cycles, bipartite graphs, and coalescences of graphs. In
particular, we show that for some families of dense graphs (e.g.~complete
bipartite graphs), every power greater than $1$ preserves positivity on
$\bp_G([0,\infty))$. This result came as a surprise since, as shown by
FitzGerald and Horn \cite{FitzHorn}, non-integer powers smaller than
$n-2$ do not preserve positivity on $n \times n$ matrices when no
additional structure of zeros is imposed. The result also has important
implications for the regularization of covariance/correlation matrices,
by showing that small powers can be safely used to regularize
covariance/correlation matrices having an appropriate original structure
of zeros. 

The remainder of the paper is structured as follows: the key definitions
and the main theorem of the paper are introduced in the rest of Section
\ref{Sintro}. Useful preliminary results are discussed in Section
\ref{Sprelim}. The main theorem is proved in Section \ref{Smain},
followed by a study of non-chordal graphs in Section \ref{Snon-chordal}.
We conclude by discussing further questions and extensions.

\subsection{Main result}\label{S2}

In order to state our main theorem, we begin by introducing some
notation. Given two $n \times n$ matrices $A = (a_{ij})$ and $B =
(b_{ij})$, their Hadamard (or Schur, or entrywise) product, denoted by $A
\circ B$, is defined by $A \circ B  := (a_{ij} b_{ij})$. Note that $A
\circ B$ is a principal submatrix of the tensor product $A \otimes B$. As
a consequence, if $A$ and $B$ are positive (semi)definite, then so is $A
\circ B$.  This result is known in the literature as the \textit{Schur
product theorem} \cite{Schur1911}. Given $\alpha \in \R$, we denote the
entrywise $\alpha$th power of a matrix $A$ with nonnegative entries by
$A^{\circ \alpha} := (a_{ij}^\alpha)$, where we define $0^\alpha := 0$
for all $\alpha$. As a consequence of the Schur product theorem, it is
clear that $A^{\circ k}$ is positive (semi)definite for all positive
(semi)definite matrices $A$ and all  $k \in \N$. Note that $A^{\circ
\alpha}$ is not always well-defined if $a_{ij} \in \R$.
Following Hiai \cite{Hiai2009}, it is natural to replace the power
functions by their odd and even extensions to $\R$ in order to deal with
arbitrary positive semidefinite matrices. Given $\alpha \in \R$, we
define the odd and even extensions of the power functions as follows:
\begin{equation}
\psi_\alpha(x) := {\rm sgn}(x) |x|^\alpha, \qquad \phi_\alpha(x) :=
|x|^\alpha, \qquad \forall\ x \in \R \setminus \{ 0 \},
\end{equation}

\noindent and $\psi_\alpha(0) = \phi_\alpha(0)  := 0$. Given $f: \R \to
\R$, and $A = (a_{ij})$, define $f[A] := (f(a_{ij}))$. We now introduce
the main objects of study in this paper.  

\begin{definition}\label{D1}
Let $n \geq 2$ and let $G = (V,E)$ be a simple graph on $V =
\{1,\dots,n\}$. We define: 
\begin{align*}
\calh_G := &\ \{ \alpha \in \R : A^{\circ \alpha} \in \bp_G \text{ for
all } A \in \bp_G([0,\infty)) \},\\
\calh_G^\psi := &\ \{ \alpha \in \R : \psi_\alpha[A] \in \bp_G \text{ for
all } A \in \bp_G(\R) \}, \\
\calh_G^\phi := &\ \{ \alpha \in \R : \phi_\alpha[A] \in \bp_G \text{ for
all } A \in \bp_G(\R) \}.
\end{align*}
\end{definition}

In particular, observe that if $H \subset G$ is a subgraph, then $\calh_H
\supset \calh_G$. Moreover, $\calh_G$ is closed for all graphs $G$.

Denote by $K_n$ the complete graph on $n$ vertices. The sets
$\calh_{K_n}, \calh_{K_n}^\psi$, and $\calh_{K_n}^\phi$ were computed
through several papers and the following theorem summarizes their
results. The reader is referred to \cite{GKR-crit-2sided} for more
details. 

\begin{theorem}[FitzGerald--Horn \cite{FitzHorn}, Bhatia--Elsner
\cite{Bhatia-Elsner}, Hiai \cite{Hiai2009}, Guillot--Khare--Rajaratnam
\cite{GKR-crit-2sided}]\label{Tcomplete}
Let $n \geq 2$. The $\calh$-sets of powers preserving positivity for $G =
K_n$ are: 
\begin{align*}
\calh_{K_n} = &\ \N \cup [n-2,\infty), \\
\calh_{K_n}^\psi = &\ (-1+2\N) \cup [n-2,\infty), \\
\calh_{K_n}^\phi = &\ 2\N \cup [n-2, \infty).
\end{align*}
\end{theorem}

The above surprising result shows that there is a threshold value above
which every power function $x^\alpha, \psi_\alpha$, or $\phi_\alpha$
preserves positivity on $\bp_n([0,\infty))$ or $\bp_n(\R)$, when applied
entrywise. The threshold is commonly referred to as the \textit{critical
exponent} for preserving positivity. It is natural to extend the notion
of \textit{critical exponents} to arbitrary graphs.

\begin{definition}
Given a graph $G$, define the \emph{Hadamard critical exponents of $G$}
to be
\begin{align*}
CE_H(G) := &\ \min \{\alpha \in \R : A \in \bp_G([0,\infty)) \Rightarrow
A^{\circ \beta} \in \bp_G \textrm{ for every } \beta \geq \alpha\},\\
CE_H^\psi(G) := &\ \min \{\alpha \in \R : A \in \bp_G(\R) \Rightarrow
\psi_\alpha[A] \in \bp_G \textrm{ for every } \beta \geq \alpha\},\\
CE_H^\phi(G) := &\ \min \{\alpha \in \R : A \in \bp_G(\R) \Rightarrow
\phi_\alpha[A] \in \bp_G \textrm{ for every } \beta \geq \alpha\}.
\end{align*}
\end{definition}

\noindent Note that since every graph $G = (V,E)$ is contained in a
complete graph, the critical exponents of $G$ are well defined by Theorem
\ref{Tcomplete}, and bounded above by $|V|- 2$. 

We can now state our main result. Let $K_n^{(1)}$ denote the complete
graph on $n$ vertices with one edge missing. 

\begin{theorem}[Main result]\label{Tmain}
Let $G$ be any chordal graph with at least $2$ vertices and let $r$ be
the largest integer such that either $K_r$ or $K_r^{(1)}$ is 
a subgraph of $G$. Then
\begin{align*}
\calh_G &= \N \cup [r-2,\infty), \\
\calh_G^\psi &= (-1+2\N) \cup [r-2, \infty), \\
\calh_G^\phi &= 2\N \cup [r-2, \infty). 
\end{align*}
In particular, $CE_H(G) = CE_G^\psi(G) = CE_G^\phi(G) = r-2$. 
\end{theorem}

Theorem \ref{Tmain} thus demonstrates that an increase in sparsity
generally reduces the Hadamard critical exponents. The precise way in
which the critical exponents of chordal graphs are lowered is driven by
the size of their largest maximal or nearly maximal cliques. This fact is
especially important in applications, where covariance/correlation
matrices need to be regularized by minimally modifying their entries
while simultaneously preserving positive semidefiniteness.
Theorem \ref{Tmain} shows that small powers can be used to achieve such a
goal if the original matrices are sparse enough.

\begin{remark}
Theorem \ref{Tmain} shows that the critical exponent of a chordal graph
$G$ is bounded above by $\max_{v \in V(G)} \deg(v) - 1$, as well as by
$\omega(G) - 1 = tw(G)$, where $\omega(G)$ and $tw(G)$ denote the clique
number and treewidth of $G$ respectively. Note however that these bounds
are not sharp. For instance, for star graphs the critical exponent is
always $1$ (see Theorem \ref{Taddtrees}), and for complete graphs the
critical exponent is $\omega(G)-2$. Also note that for all graphs $G$, we
have $\calh_G^\psi, \calh_G^\phi \subset \calh_G$, and $r-2$ provides a
lower bound for the critical exponents of $G$, where $K_r$ or $K_r^{(1)}$
is a subgraph of $G$.
\end{remark}

The rest of the paper is devoted to proving Theorem \ref{Tmain}. Most of
the techniques and constructions that are traditionally used to study
powers preserving positivity (e.g.~spectral methods, Schur complements)
do not preserve the structure of zeros of matrices. Studying powers
preserving positivity under sparsity constraints is thus a challenging
task that requires new ideas. In the rest of the paper, we develop
multiple techniques for computing the $\calh$-sets of graphs. In addition
to proving Theorem \ref{Tmain}, we use these techniques to compute the
critical exponent of many non-chordal graphs as well. We demonstrate that
the critical exponent does not always correspond to the order of the
largest induced $K_r$ or $K_r^{(1)}$ minus $2$ when $G$ is non-chordal.
We also show that many dense graphs (e.g. complete bipartite graphs) have
a surprisingly small critical exponent that does not depend on their
number of vertices. This is in stark contrast to the family of complete
graphs $K_n$, for which the critical exponent is $n-2$.

\section{Preliminary results: pendant edges and trees}\label{Sprelim}

We begin our analysis by recalling a general result that classifies
entrywise functions preserving positivity for matrices with zeros
according to a tree. Given a $n \times n$ symmetric matrix $A = (a_{ij})$
and a graph $G = (V,E)$ with vertex set $V = \{1,\dots,n\}$ and edge set
$E$, denote by 
\begin{equation*}
f_G[A] := \begin{cases}
f(a_{ij}) & \textrm{ if } (i,j) \in E \textrm{ or } i=j, \\
0 & \textrm{otherwise}.
\end{cases}
\end{equation*}

\begin{theorem}[Guillot, Khare, and Rajaratnam,
\cite{Guillot_Rajaratnam2012b}]\label{Tsparse}
Suppose $I = [0,R)$ for some $0 < R \leq \infty$, and $f : I \to
[0,\infty)$. Let $G$ be a tree with at least $3$ vertices, and let $P_3$
denote the path on $3$ vertices. Then the following are equivalent:
\begin{enumerate}
\item $f_G[A] \in \bp_G$ for every $A \in \bp_G(I)$;
\item $f_T[A] \in \bp_T$ for all trees $T$ and all matrices $A \in
\bp_T(I)$;
\item $f_{P_3}[A] \in \bp_{P_3}$ for every $A \in \bp_{P_3}(I)$; 
\item The function $f$ satisfies:
\begin{equation}\label{Emidconvex}
f(\sqrt{xy})^2 \leq f(x) f(y), \qquad \forall x,y \in I
\end{equation}
and is superadditive on $I$, i.e., 
\begin{equation}\label{Esuper}
f(x+y) \geq f(x) + f(y), \qquad \forall x,y,x+y \in I.
\end{equation}
\end{enumerate}
\end{theorem}

It follows from Theorem \ref{Tsparse} that $\calh_G = [1,\infty)$ for any
tree $G$. We now generalize this result to graphs obtained by pasting
trees to vertices of graphs, and to the functions $\psi_\alpha$ and
$\phi_\alpha$.

\begin{theorem}\label{Taddtrees}
Suppose $G'$ is not a disjoint union of copies of $K_2$. Construct a
graph $G$ from $G'$ by attaching finitely many (possibly empty) trees to
each node, at least one of which is not isolated. Then 
\begin{equation*}
\calh_G = \calh_{G'}, \quad  \calh^\psi_G = \calh^\psi_{G'},\textrm{ and}
\quad \calh^\phi_G = \calh^\phi_{G'}.
\end{equation*}
In particular, if $G$ is a tree with at least $3$ vertices, then $CE_H(G)
= CE_H^\phi(G) = CE_H^\psi(G) = 1$.
\end{theorem}

In order to prove Theorem \ref{Taddtrees}, we first introduce additional
notation.

\begin{definition}
Let $A,B$ be two $n \times n$ matrices and $1 \leq i \leq n$. Write $A$
in block
form:
\[ A = \begin{pmatrix} A_{11} & {\bf u}_1 & A_{12}\\{\bf v}_1^T & a_{ii}
& {\bf v}_2^T\\A_{21} & {\bf u}_2 & A_{22} \end{pmatrix}. \]

\noindent If $a_{ii} \neq 0$, then the \textit{Schur complement} of
$a_{ii}$ in $A$ is defined to be
\[ A / a_{ii} := \begin{pmatrix} A_{11} & A_{12}\\A_{21} & A_{22}
\end{pmatrix} - a_{ii}^{-1} \begin{pmatrix} {\bf u}_1\\{\bf u}_2
\end{pmatrix} ({\bf v}_1^T\ {\bf v}_2^T) . \]

\noindent We also say that $A \geq B$ if $A-B \in \bp_n$.
\end{definition}

\begin{proof}[Proof of Theorem \ref{Taddtrees}]
We only prove the result for $\calh_G^\psi$; the proofs are
similar for $\calh_G, \calh_G^\phi$. The first step is to compute
$\calh^\psi_{P_3}$. Note by Theorem \ref{Tcomplete} that
\[ [0,\infty) = \calh^\psi_{K_2} \supset \calh^\psi_{P_3} \supset
\calh^\psi_{K_3} = \N \cup [1,\infty), \]

\noindent so it suffices to show that no $\alpha \in (0,1)$ preserves
$\bp_{P_3}$. Now fix $a \in [0,1]$ and consider the matrix \[ A(a) :=
\begin{pmatrix} 1 & a & 0\\a & 1 & \sqrt{1-a^2}\\0 & \sqrt{1-a^2} &
1\end{pmatrix} \in \bp_{P_3}. \]

\noindent It is clear that $\psi_\alpha[A(a)] = \phi_\alpha[A(a)] =
A(a)^{\circ \alpha}$ has determinant $1 - (a^2)^\alpha - (1-a^2)^\alpha$,
and this is strictly negative if $a \in (0,1)$ and $\alpha \in [0,1)$, by
the subadditivity of $x \mapsto x^\alpha$. 

Now suppose $G'$ is a connected, nonempty graph, $v \in V(G')$, and $G$
is obtained by attaching a pendant edge to $v$ (i.e., adding a new vertex
and connecting it by an edge to $v$). Also suppose $0 \in I \subset \R$
and $f :I \to \R$ is super-additive on $I \cap [0,\infty)$. Then we claim
that $f[-]$ preserves positivity on $\bp_G(I)$ if and only if it
preserves positivity on $\bp_{G'}(I)$. The proof uses arguments from the
proof of \cite[Theorem A]{Guillot_Khare_Rajaratnam2012}.

Finally, the result for trees follows immediately by applying the
previous two steps to $I = [0,\infty)$ and $f(x) = x^\alpha$ for $\alpha
\geq 1$, and to $I = \R$ and $f(x) = \psi_\alpha(x), \phi_\alpha(x)$ for
$\alpha \geq 1$.
\end{proof}

\noindent As a consequence, we characterize all graphs with Hadamard
critical exponent $0$.

\begin{corollary}\label{CE0}
Given a graph $G$, the following are equivalent:
\begin{enumerate}
\item $G$ is a disjoint union of copies of $K_2$.
\item $G$ does not contain a connected subgraph with three vertices.
\item $\calh_G = [0,\infty)$.
\item $0 \in \calh_G$.
\item $\calh_G \nsubseteq [1,\infty)$.
\item $CE_H(G) = 0$.
\end{enumerate}
\end{corollary}

\begin{proof}
That the first two conditions are equivalent is obvious. That $(1)
\implies (3) \implies (4) \implies (5)$ is also clear. 
Now if (2) fails to hold, then this connected subgraph of $G$ is either
the path graph $P_3$ or the complete graph $K_3$. In both cases, (5) also
fails to hold, by Theorems \ref{Tcomplete} and \ref{Taddtrees}. This
shows that $(5) \implies (2)$, and hence that (1)--(5) are equivalent.
Next, clearly $(6) \implies (3)$. Conversely, if (1) holds then it is
easy to show that $\calh_G = [0,\infty)$,
so that (6) also holds.
\end{proof}

\begin{remark}
We remark that Corollary \ref{CE0} also holds if $\calh_G$ is replaced by
$\calh_G^\psi$ or $\calh_G^\phi$, and $CE_H(G)$ is replaced by the
corresponding critical exponent. The proof is similar. 
\end{remark}

\begin{remark}
Corollary \ref{CE0} shows that for all graphs that are not disjoint
unions of copies of $K_2$, the set of powers preserving positivity are
all contained in $[1,\infty)$. For this reason, and without further
reference, we do not consider non-positive entrywise power functions in
the remainder of the paper. Similarly, the Schur product theorem implies
that $\N \subset \calh_G,\ -1+2\N \subset \calh_G^\psi, \ 2\N \subset
\calh_G^\phi$ for all graphs $G$, and these facts are used below without
further reference.
\end{remark}

\section{Proof of the Main result}\label{Smain}

In this section we develop all the tools that are required to compute the
$\calh$-sets for chordal graphs. We begin by recalling some important
properties of chordal graphs (see e.g.~\cite[Chapter 5.5]{Diestel},
\cite[Chapter 4]{golumbic} for more details). 

Let $G = (V,E)$ be an undirected graph. Given $C \subset V$, denote by
$G_C$ the subgraph of $G$ \textit{induced} by $C$. A \textit{clique} in
$G$ is a complete induced subgraph of $G$. A subset $C \subset V$ is said
to \textit{separate} $A \subset V$ from $B \subset V$ if every path from
a vertex in $A$ to a vertex in $B$ intersects $C$. A partition $(A,C,B)$
of subsets of $V$ is said to be a \textit{decomposition} of $G$ if
\begin{enumerate}
\item $C$ separates $A$ from $B$; and
\item $G_C$ is complete. 
\end{enumerate}

A graph $G$ is said to be \textit{decomposable} if either $G$ is
complete, or if there exists a decomposition $(A,C,B)$ of $G$ such that
$G_{A \cup C}$ and $G_{B \cup C}$ are decomposable.

Let $G$ be a graph and let $B_1, \dots, B_k$ be a sequence of subsets of
vertices of $G$. Define: 
\begin{equation}\label{Ehistories}
H_j := B_1 \cup \dots \cup B_j, \qquad R_j = B_j \setminus H_{j-1},
\qquad S_j = H_{j-1} \cap B_j, \qquad 1 \leq j \leq k, 
\end{equation}
and $H_0 := \emptyset$. The sets $H_j, R_j$, and $S_j$ are respectively called the
\textit{histories, residuals, and separators} of the sequence. The
sequence $B_1, \dots, B_k$ is said to be a \textit{perfect ordering} if: 
\begin{enumerate}
\item For all $1 < i \leq k$, there exists $1 \leq j < i$ such that $S_i
\subset B_j$; and
\item The sets $S_i$ induce complete graphs for all $1 \leq i \leq k$. 
\end{enumerate}

Decompositions and perfect orderings provide important characterizations
of chordal graphs, as summarized in Theorem \ref{Tchordalcarac}. 

\begin{theorem}[{\cite[Propositions 2.5 and
2.17]{lauritzen}}]\label{Tchordalcarac}
Let $G = (V,E)$ be an undirected graph. Then the following are
equivalent: 
\begin{enumerate}
\item $G$ is chordal (i.e., each cycle with $4$ vertices or more in $G$
has a chord).
\item $G$ is decomposable. 
\item The maximal cliques of $G$ admit a perfect ordering.
\end{enumerate}
\end{theorem}

We now relate the decomposition of a chordal graph $G$ to properties of
functions preserving positivity on $\bp_G$. Given a graph $G$ and a
function $f : \R \to \R$ with $f(0) = 0$, we say that $f[-]$ is
\textit{Loewner super-additive on $\bp_G(\R)$} if $f[A+B] - f[A] - f[B]
\in \bp_G(\R)$ whenever $A,B \in \bp_G(\R)$. Note that this notion
coincides with the usual notion of super-additivity on $[0,\infty)$ when
$G$ has only one vertex.

\begin{theorem}\label{Tchordal}
Let $G = (V,E)$ be a graph with a decomposition $(A,C,B)$. Also let $f:
\R \to \R$.
\begin{enumerate}
\item If $f[-]$ preserves positivity on $\bp_{G_{A \cup C}}$ and on
$\bp_{G_{B \cup C}}$, and is Loewner super-additive on $\bp_{G_C}$, then
$f[-]$ preserves positivity on $\bp_G$. 
\item Conversely, if $f = \psi_\alpha$ or $f= \phi_\alpha$ and $f[-]$
preserves positivity on $\bp_G$, then $f[-]$ is Loewner super-additive on
$\bp_{G_{C'}}$ for every clique $C' \subset C$ for which there exist
vertices $v_1 \in A, v_2 \in B$ that are adjacent to every $v \in C'$.
\end{enumerate}

\noindent In particular, when $f = \psi_\alpha$ or $f= \phi_\alpha$ and
$|C| = 1$, $f[-]$ preserves positivity on $\bp_G$, if and only if $f[-]$
preserves positivity on $\bp_{G_{A \cup C}}$ and $\bp_{G_{B \cup C}}$ and
is Loewner super-additive on $[0,\infty)$.
\end{theorem}

\noindent Theorem \ref{Tchordal} immediately implies that if a
superadditive function preserves positivity on $\bp_2$, then it does so
on $\bp_G$ for all trees $G$. The result thus extends \cite[Theorem
A]{Guillot_Khare_Rajaratnam2012}. (See \cite[Theorem
2.6]{Guillot_Khare_Rajaratnam2012} for a characterization of entrywise
functions preserving positivity on $\bp_2$.)

The proof of Theorem \ref{Tchordal} requires some preliminary results. We
first recall previous work on Loewner superadditive functions. The powers
that are Loewner superadditive on $\bp_n(\R) = \bp_{K_n}(\R)$ have been
classified in \cite{GKR-crit-2sided}.

\begin{theorem}[{Guillot, Khare, and Rajaratnam \cite[Theorem
5.1]{GKR-crit-2sided}}]\label{Tsuperadd}
Given an integer $n \geq 2$, the sets of entrywise power functions
$x^\alpha, \psi_\alpha, \phi_\alpha$ (with $\alpha \in \R$) which are
Loewner super-additive maps on $\bp_n$ are, respectively,
\[ \N \cup [n, \infty), \qquad
(-1+2\N) \cup [n, \infty), \qquad
2\N \cup [n, \infty). \]

\noindent Moreover, for all $\alpha \in (0,n) \setminus \N$, there exist
${\bf u}, {\bf v} \in [0,\infty)^n$ such that $({\bf u} {\bf u}^T + {\bf
v} {\bf v}^T)^{\circ \alpha} \not\in \bp_n$.
Similarly, if $f \equiv \psi_\alpha$ with $\alpha = 2 k$ for $1 \leq k
\leq \lceil n/2 \rceil - 1$, or $f \equiv \phi_\alpha$ with $\alpha = 2
k-1$ for $1 \leq k \leq \lfloor n/2 \rfloor$, then there exist
${\bf u}, {\bf v} \in \R^n$ such that $f[{\bf u} {\bf u}^T + {\bf v} {\bf
v}^T] \not\in \bp_n$.
\end{theorem}

The following corollary is an immediate consequence of Theorem
\ref{Tchordal} and Theorem \ref{Tsuperadd}.

\begin{corollary}\label{Cchordal}
Let $G = (V,E)$ be a graph with a decomposition $(A,C,B)$. Suppose there
exist vertices $v_1 \in A$ and $v_2 \in B$ that are adjacent to every $v
\in C$. Let $f = \psi_\alpha$ or $f = \phi_\alpha$ for some $\alpha \in
\R$. Then $f[-]$ preserves positivity on $\bp_G$ if and only if either
\begin{enumerate}
\item $\alpha \in -1+2\N$ if $f = \psi_\alpha$ or $\alpha \in 2\N$ if $f
= \phi_\alpha$, or 
\item $f[-]$ preserves positivity on $\bp_{G_{A \cup C}}$ and $\bp_{G_{B
\cup C}}$ and $|\alpha| \geq |C|$.
\end{enumerate}
\end{corollary}

Lemma \ref{Ldecomp_chord} below provides an important decomposition that
will be crucial in the proof of Theorem \ref{Tchordal}. In the statement
of the result and the remainder of the paper, we adopt the following
convention to simplify notation: given a graph $G$ and an induced
subgraph $G'$, we identify $\bp_{G'}(I)$ with a subset of $\bp_G(I)$ when
convenient, via the assignment $M \ \mapsto \ M \ \oplus \ {\bf 0}_{(V(G)
\setminus V(G')) \times (V(G) \setminus V(G'))}$.

\begin{lemma}\label{Ldecomp_chord}
Let $G = (V,E)$ be a graph with a decomposition $(A,C,B)$ of $V$, and let
$M$ be a symmetric matrix. Assume the principal submatrices $M_{AA}$ and
$M_{BB}$ of $M$ are invertible. Then the following are equivalent: 
\begin{enumerate}
\item $M \in \bp_G$.
\item $M = M_1 + M_2$ for some matrices $M_1 \in \bp_{G_{A \cup C}}$ and
$M_2 \in \bp_{G_{B \cup C}}$. 
\end{enumerate} 
\end{lemma}

\begin{proof}
Clearly $(2) \implies (1)$. Now let $M \in \bp_G$. The matrix $M$ can be
written in block form as
\begin{equation*}
M = \begin{pmatrix}
M_{AA} & M_{AC} & 0 \\
M_{AC}^T & M_{CC} & M_{CB} \\
0 & M_{CB}^T & M_{BB}
\end{pmatrix}.
\end{equation*}

It is not
difficult to verify that
\begin{equation}\label{Emiracle_decomposition}
M = 
\begin{pmatrix}
M_{AA} & 0 & 0 \\
M_{AC}^T & \Id_{|C|} & M_{CB} \\
0 & 0 & M_{BB}
\end{pmatrix}
\begin{pmatrix}
M_{AA}^{-1} & 0 & 0 \\
0 & S & 0\\
0 & 0 & M_{BB}^{-1}
\end{pmatrix}
\begin{pmatrix}
M_{AA} & 0 & 0 \\
M_{AC}^T & \Id_{|C|} & M_{CB} \\
0 & 0 & M_{BB}
\end{pmatrix}^T, 
\end{equation}
where $\Id_k$ denotes the $k \times k$ identity matrix, and
$S := M_{CC} - M_{AC}^T M_{AA}^{-1} M_{AC} - M_{CB} M_{BB}^{-1}
M_{CB}^T$. It follows that $S$ is positive semidefinite. Now let 
\begin{equation*}
M_1 := \begin{pmatrix}
M_{AA} & M_{AC} & 0 \\
M_{AC}^T & M_{AC}^T M_{AA}^{-1} M_{AC} & 0 \\
0 & 0 & 0
\end{pmatrix}, \qquad M_2 := \begin{pmatrix}
0 & 0 & 0 \\
0 & M_{CC} - M_{AC}^T M_{AA}^{-1} M_{AC} & M_{CB} \\
0 & M_{CB}^T & M_{BB}
\end{pmatrix}.
\end{equation*}

\noindent Clearly, $M = M_1 + M_2$. Computing the Schur complement of
$M_{AA}$ in the upper left $2 \times 2$ blocks of $M_1$, we conclude 
 that $M_1 \in \bp_{G_{A \cup C}}$. Similarly, the Schur complement of
the lower right $2 \times 2$ blocks of $M_2$ is equal to $S$ and
therefore $M_2 \in \bp_{G_{B \cup C}}$. This proves the desired
decomposition of $M$.
\end{proof}

Using the above results, we now prove Theorem \ref{Tchordal}.

\begin{proof}[Proof of Theorem \ref{Tchordal}]
Suppose $f$ satisfies the conditions in (1), and $M \in \bp_G$. Then, in
particular, $f$ preserves positivity on $\bp_{K_2}(\R) = \bp_2(\R)$,
whence $f$ is continuous on $(0,\infty)$ by \cite[Theorem 1.2]{Horn}.
Moreover, $f$ is right-continuous at $0$ as shown at the beginning of the
proof of \cite[Theorem C]{Guillot_Khare_Rajaratnam2012}. Also note that
$f(0) = 0$ because $f[-]$ is super-additive on $[0,\infty) = \bp_1
\subset \bp_{G_C}$ and preserves positivity on $\bp_1 \subset \bp_{G_{A
\cup C}}$.
Now given $\epsilon > 0$, let $\widetilde{M}_\epsilon := M + \epsilon
\cdot \Id_{|G|}$.
By Lemma \ref{Ldecomp_chord}, $\widetilde{M}_\epsilon = M_1 + M_2$ with
$M_1 \in \bp_{G_{A \cup C}}$ and $M_2 \in \bp_{G_{B \cup C}}$. By
assumption $f[M_1]$ and $f[M_2]$ are positive semidefinite.
Moreover, $D := f[\widetilde{M}_\epsilon] - f[M_1] - f[M_2]$ belongs to
$\bp_{G_C}$ by the assumption of superadditivity on $f$ on $\bp_{G_C}$.
It follows that $f[\widetilde{M}_\epsilon] = f[M_1] + f[M_2] + D$ is
positive semidefinite for every $\epsilon > 0$. We conclude by continuity
that $f[M]$ is positive semidefinite, proving (2).

Next, suppose $f = \psi_\alpha$ or $\phi_\alpha$ for $\alpha \in \R$, and
$f[-]$ preserves positivity on $\bp_G$. Then clearly $f[-]$ preserves
positivity on $\bp_{G_{A \cup S}}$ and $\bp_{G_{B \cup S}}$. Moreover,
suppose there exist $v_1 \in A$, $v_2 \in B$, and a clique $C' \subset C$
of size $m$ such that $v_1$ and $v_2$ are adjacent to every vertex in
$C'$. Assume, without loss of generality, that the vertices of $G$ as
labelled in the following order: $v_1$, the $m$ vertices in $C'$, $v_2$,
and the remaining vertices of $G$. Now given vectors ${\bf u}, {\bf v}
\in \R^m$ and a $m \times m$ symmetric matrix $M$, define the matrix
\begin{equation}\label{Enk1}
W({\bf u}, {\bf v},M) := \begin{pmatrix} 1 & {\bf u}^T & 0\\ {\bf u} & M
& {\bf v}\\ 0 & {\bf v}^T & 1 \end{pmatrix}.
\end{equation}

\noindent Then $W({\bf u}, {\bf v}, {\bf u} {\bf u}^T + {\bf v} {\bf
v}^T) \oplus {\bf 0}_{|V| - (m+2)} \in \bp_G(\R)$, so by the
assumptions on $f$, we conclude that $f[W({\bf u}, {\bf v}, {\bf u} {\bf
u}^T + {\bf v} {\bf v}^T)] = W(f[{\bf u}], f[{\bf v}], f[{\bf u} {\bf
u}^T + {\bf v} {\bf v}^T]) \in \bp_{m+2}(\R)$. 
%
 Now using the same decomposition as in
\eqref{Emiracle_decomposition}, we conclude that 
\begin{equation}\label{Esuper_rank1}
f[{\bf u} {\bf u}^T + {\bf v} {\bf v}^T] - f[{\bf u}]f[{\bf u}^T] -
f[{\bf v}]f[{\bf v}^T] = f[{\bf u} {\bf u}^T + {\bf v} {\bf v}^T] -
f[{\bf uu^T}] - f[{\bf vv}^T] \geq 0.
\end{equation}

\noindent Thus $f = \psi_\alpha, \phi_\alpha$ is Loewner super-additive
on rank one matrices in $\bp_m$. By Theorem \ref{Tsuperadd}, the Loewner
super-additive powers preserving positivity on rank $1$ matrices are the
same as the Loewner super-additive powers. We therefore conclude that $f$
is Loewner super-additive on all of $\bp_m$.
\end{proof}

We now have all the ingredients to prove the main result of the paper.

\begin{proof}[Proof of Theorem \ref{Tmain}]
Before proving the result for all chordal graphs, let us prove it for the
``nearly complete'' graphs $K_r^{(1)}$. The result is obvious for $r=2$.
Now suppose $r \geq 3$. First note that $K_{r-1} \subset K_r^{(1)}
\subset K_r$, so
\[ 2\N \cup [r-2,\infty) = \calh^\phi_{K_r} \subset
\calh^\phi_{K_r^{(1)}} \subset \calh^\phi_{K_{r-1}} = 2\N \cup
[r-3,\infty). \]

\noindent Similarly, we have $(-1 + 2\N) \cup [r-2,\infty) \subset
\calh^\psi_{K_r^{(1)}} \subset (-1+2\N) \cup [r-3,\infty)$. Now label the
vertices from $1$ to $r$ such that $(1,r) \not\in E(K_r^{(1)})$, and
apply Corollary \ref{Cchordal} with $A = \{ 1 \}, B = \{ r \}$, and $S =
\{ 2, \dots, r-1 \}$. It follows immediately that
\[ \calh^\phi_{K_r^{(1)}} = 2\N \cup [r-2,\infty), \qquad
\calh^\psi_{K_r^{(1)}} = (-1+2\N) \cup [r-2,\infty). \]

\noindent Finally, $\calh_{K_r} = \N \cup [r-2,\infty) \subset
\calh_{K_r^{(1)}}$. To show the reverse inclusion, suppose $x^\alpha$
preserves $\bp_{K_r^{(1)}}([0,\infty))$. Given ${\bf u}, {\bf v} \in
[0,\infty)^{r-2}$ and $M \in \bp_{r-2}([0,\infty))$, define $W({\bf u},
{\bf v}, M)$ as in \eqref{Enk1}. Then $W({\bf u}, {\bf v}, {\bf u} {\bf
u}^T + {\bf v} {\bf v}^T) \in \bp_{K_r^{(1)}}([0,\infty))$, so
\begin{equation*}
W({\bf u}, {\bf v}, {\bf u} {\bf u}^T + {\bf v} {\bf v}^T)^{\circ \alpha}
\in \bp_{K_r^{(1)}}([0,\infty)), \qquad \forall {\bf u}, {\bf v} \in
[0,\infty)^{r-2}. 
\end{equation*}

\noindent Proceeding as in \eqref{Esuper_rank1}, we conclude that the
entrywise function $x \mapsto x^\alpha$ is Loewner super-additive on rank
one matrices in $\bp_{r-2}([0,\infty))$. Thus $\alpha \in \N$ or $\alpha
\geq r-2$ by Theorem \ref{Tsuperadd}. It follows that $\calh_{K_r^{(1)}}
= \N \cup [r-2,\infty)$. This proves the theorem for $G = K_r^{(1)}$. 

Now suppose $G$ is an arbitrary chordal graph, which without loss of
generality we assume to be connected. Denote by $r$ the largest integer
such that $G$ contains $K_r$ or $K_r^{(1)}$ as an induced subgraph. By
the above calculation, 
\begin{equation}\label{Einclusion}
\calh_G \subset \N \cup [r-2,\infty), \qquad \calh_G^\psi \subset
(-1+2\N) \cup [r-2,\infty), \qquad \calh_G^\phi = 2\N \cup [r-2,\infty).
\end{equation}

\noindent We now prove the reverse inclusions. By Theorem
\ref{Tchordalcarac}, the maximal cliques of $G$ admit a perfect ordering
$\{C_1, \dots, C_k\}$. We will prove the reverse inclusions in
\eqref{Einclusion} by induction on $k$. If $k=1$, then $G$ is complete
and the inclusions clearly hold by Theorem \ref{Tcomplete}. Suppose the
result holds for all chordal graphs with $k=l$ maximal cliques, and let
$G$ be a graph with $k = l+1$ maximal cliques. For $1 \leq j \leq k$,
define 
\begin{equation}\label{EhistoriesC}
H_j := C_1 \cup \dots \cup C_j, \qquad C_j = C_j \setminus H_{j-1},
\qquad S_j = H_{j-1} \cap C_j
\end{equation}

\noindent as in \eqref{Ehistories}. By \cite[Lemma 2.11]{lauritzen}, the
triplet $(H_{k-1}, S_k, R_k)$ is a decomposition of $G$. Let $r$ be the
largest integer such that $G$ contains $K_r$ or $K_r^{(1)}$ as an induced
subgraph, and let $\alpha \in [r-2,\infty)$. By the induction hypothesis,
the three $\alpha$-th power functions preserve positivity on
$\bp_{G_{H_{k-1} \cup S_k}} = \bp_{G_{H_{k-1}} }$.
Moreover, since $\alpha \geq r-2$, they also preserve positivity on
$\bp_{G_{C_k \cup S_k}} = \bp_{G_{C_k}}$. We now claim that $r \geq |S_k|
+ 2$. Clearly, $|S_k| \leq r$ since $S_k$ is complete. If $|S_k| = r$,
then $C_k$ is contained in one of the previous cliques, which is a
contradiction.
Suppose instead that $|S_k| = r-1$. Since $\{C_1, \dots, C_k\}$ is a
perfect ordering, $S_k \subset C_i$ for some $i < k$. Let $v \in C_i
\setminus S_k$ and let $w \in R_k$. Note that both $v$ and $w$ are
adjacent to every $s \in S_k$. Thus, the subgraph of $G$ induced by $S_k
\cup \{v,w\}$ is isomorphic to $K_{r+1}^{(1)}$, which contradicts the
definition of $r$. We therefore conclude that $r \geq |S_k| + 2$, as
claimed. As a consequence, the $\alpha$-th power functions are Loewner
super-additive on $\bp_{S_k}$ by Theorem \ref{Tsuperadd}. Applying
Theorem \ref{Tchordal}, we conclude that $\alpha \in \calh_G^{\psi},
\calh_G^\phi$. Since $\calh_G^\psi \cup \calh_G^\phi \subset \calh_G$, we
obtain that $\alpha \in \calh_G$ as well. This concludes the proof of the
theorem. 
\end{proof}

\begin{remark}\label{Ralternative_crit}
The critical exponent of a chordal graph $G$ can also be defined as
$\max(c-2,s)$, where $c = \omega(G)$ is the clique number of $G$, and $s$
is the size of the largest separator associated to a perfect clique
ordering of $G$ (see \eqref{EhistoriesC}). This follows from the proof of
Theorem \ref{Tmain} where it was shown that if such a separator has size
$s$, then either $s \leq c-2$ or $G$ contains $K_{s+2}^{(1)}$ as an
induced subgraph. The critical exponent can also be computed by replacing
$s$ by the size of the largest intersection of two maximal cliques, as
shown in Corollary \ref{Cformula} below. 
\end{remark}

We now mention several consequences of the above analysis in this
section. The following corollary provides a formula that can be used to
systematically compute the critical exponent of a chordal graph.

\begin{corollary}\label{Cformula}
Suppose $G = (V,E)$ is chordal with $V = \{ v_1, \dots, v_m \}$, and let
$C_1, \dots, C_n$ denote the maximal cliques in $G$. Define the $m \times
n$ ``maximal clique matrix'' $M(G)$ of $G$ to be $M(G) := ({\bf 1}(v_i
\in C_j))$, i.e.,  
\begin{equation*}
M(G)_{ij}  = \begin{cases}
1 & \textrm{ if } v_i \in C_j, \\ 
0 & \textrm{ otherwise}.
\end{cases}
\end{equation*}
Let $u_1, \dots, u_n \in \{ 0, 1 \}^m$ denote the columns of $M(G)$. Then
the critical exponent of $G$ is given by
\begin{equation}\label{Eformula}
CE_H(G) = CE^\psi_H(G) = CE^\phi_H(G) = \max_{i,j} (u_i^T u_j - 2
\delta_{i,j}),
\end{equation}

\noindent i.e., the largest entry of $M(G)^T M(G) - 2 \Id_{|V|}$.
\end{corollary}

\begin{proof}
Let $c$ and $s$ denote the maximum of the diagonal and off-diagonal
entries of $M(G)^T M(G) - 2 \Id_{|V|}$ respectively. Clearly, $c$ is the
size of the maximal cliques of $G$ minus $2$ and $s = \max_{i \ne j} |C_i
\cap C_j|$. By Theorem \ref{Tchordalcarac}, the cliques of $G$ admit a
perfect ordering, say,  $C_{i_1}, \dots, C_{i_n}$. For $i \ne j$, let $k,
l$ be such that $C_i = C_{i_k}$ and $C_j = C_{i_l}$. Without loss of
generality, assume $i_k < i_l$. Then $C_i \cap C_j = C_{i_k} \cap C_{i_l}
\subset H_{i_l -1} \cap C_{i_l} = S_{i_l}$, where our notation is as in
\eqref{EhistoriesC}.
Thus, $s \leq \max_{j=1, \dots, n} |S_{i_j}|$. Conversely, since
$C_{i_1}, \dots, C_{i_n}$ is a perfect ordering, for every $1 \leq j \leq
n$, we have $S_{i_j} \subset C_{i_{j'}}$ for some $i_{j'} < i_j$. Thus,
$S_{i_j} \subset C_{i_j} \cap C_{i_{j'}}$ and so $s \geq \max_{j=1,
\dots, n} |S_{i_j}|$. It follows that $s$ corresponds to the order of the
largest separator in the perfect ordering of the cliques of $G$. We
conclude by Theorem \ref{Tmain} and Remark \ref{Ralternative_crit} that
that the critical exponents of $G$ correspond to the maximal entry of
$M(G)^T M(G) - 2 \Id_{|V|}$. 
\end{proof}

For completeness, we remark that Theorem \ref{Tchordal} also has the
following consequence for general entrywise maps. The proof is similar to
that of Theorem \ref{Tmain}.

\begin{corollary}\label{Tchordal_gen}
Let $G$ be a chordal graph, and let $\{ C_1, \dots, C_k \}$ be a perfect
ordering of its maximal cliques. Define
\begin{equation*}
c := \max_{i=1,\dots, k} |C_i| = \omega(G),
\qquad s := \max_{i=1, \dots, k} |S_i|, 
\end{equation*}

\noindent where $S_i$ is defined as in \eqref{EhistoriesC}. If
$f : \R \to \R$ is such that $f[-]$ preserves positivity on $\bp_{K_c}$
and is Loewner super-additive on $\bp_{K_s}$, then $f[-]$ preserves
positivity on $\bp_G$.
\end{corollary}

Note that Corollary \ref{Tchordal_gen} uses a clique ordering of the
vertices of a chordal graph $G$. A natural parallel approach in studying
functions preserving positivity is to build the graph $G$ step by step by
using a perfect ordering of the vertices. The following proposition
formalizes this procedure.

\begin{definition}
Given a graph $G$ on a vertex set $V$, denote by $N(v)$ the
\textit{neighborhood} of a vertex $v \in V$, i.e., $N(v) = \{w \in V :
(v,w) \in E\}$. A vertex $v \in V$ is said to be \textit{simplicial} if
$N(v) \cup \{v\}$ is complete. An ordering $\{v_1, \dots, v_n\}$ of the
vertices of $V$ is said to be a \textit{perfect elimination ordering} if
for all $i=1, \dots, n$, $v_i$ is simplicial in the subgraph of $G$
induced by $\{v_1, \dots, v_i\}$. 
\end{definition}

\begin{proposition}
Let $G$ be a chordal graph with a perfect elimination ordering of its
vertices $\{ v_1, \dots, v_n \}$. For all $1 \leq k \leq n$, denote by
$G_k$ the induced subgraph on $G$ formed by $\{ v_1, \dots, v_k \}$, so
that the neighbors of $v_k$ in $G_k$ form a clique. Define $c =
\omega(G)$ to be the clique number of $G$, and
\begin{equation*}
d := \max_{k=1, \dots, n} \deg_{G_k}(v_k).
\end{equation*}

\noindent If $f : \R \to \R$ is any function such that $f[-]$ preserves
positivity on $\bp_c^1(\R)$ and $f[M+N] \geq f[M] + f[N]$ for all $M
\in \bp_d(\R)$ and $N \in \bp_d^1(\R)$, then $f[-]$ preserves positivity
on $\bp_G(\R)$.
(Here, $\bp_c^1$ denotes the matrices in $\bp_c$ of rank at most one.)
\end{proposition}

As an illustration, if $G$ is a tree, then $c=2$ and $d=1$. Thus the
result extends \cite[Theorem A]{Guillot_Khare_Rajaratnam2012} to
arbitrary chordal graphs, with weakened hypotheses.

\begin{proof}
First note that $f(0) = 0$ since $f$ is nonnegative and super-additive on
$[0,\infty)$ by assumption. We now prove the result for $G_k$ by
induction on $k$. Clearly the result holds for $k=1$. Now suppose the
result holds for $k$. Assume without loss of generality that the
neighbors of $v_{k+1} \in V(G)$ are $v_1, \dots, v_l$ for some $1 \leq l
\leq k$, which are all adjacent to one another. Now write a matrix $A \in
\bp_{G_{k+1}}(\R)$ in the following block form, and also define an
associated matrix $U(A)$:
\[ A = \begin{pmatrix} P & Q & {\bf u}\\ Q^T & R & {\bf 0}\\ {\bf u}^T &
{\bf 0}^T & a \end{pmatrix}, \qquad U(A) := \begin{pmatrix} a^{-1} {\bf
u} {\bf u}^T & {\bf u}\\ {\bf u}^T & a \end{pmatrix}, \]

\noindent where $Q$ is $l \times (k-l)$, and we may assume that $a>0$.
Note that if $f(a) = 0$, then applying $f$ entrywise to the submatrix
$U(A) \oplus {\bf 0}_{(k-l) \times (k-l)} \in \bp_{G_{k+1}}^1(\R)$ (by
abuse of notation) shows that $f[{\bf u}] = {\bf 0}$. Hence $f[A] \in
\bp_{G_{k+1}}(\R)$ by the induction hypothesis for $G_k$.

Now suppose $f(a) > 0$. It suffices to show that the Schur complement
$S_{f[A]}$ of $f[A]$ with respect to $f(a)$ is also positive
semidefinite. Note that the Schur complement $S_A$ of $A$ with respect to
$a$ belongs to $\bp_{G_k}(\R)$. Therefore by the induction hypothesis,
$f[S_A]$ is also positive semidefinite. Thus it suffices to show that
$S_{f[A]} - f[S_A] \geq 0$. Now compute:
\[ S_{f[A]} - f[S_A] = \begin{pmatrix} f[P] - f(a)^{-1} f[{\bf u}] f[{\bf
u}]^T - f[P - a^{-1} {\bf u} {\bf u}^T] & {\bf 0}\\ {\bf 0}^T & {\bf 0}
\end{pmatrix}. \]

\noindent Next, note that $c \geq l+1$ since the subgraph of $G$ induced
by $\{ v_1, \dots, v_l, v_{k+1} \}$ is complete. Moreover, since $U(A)
\in \bp_{k+1}^1(\R)$, it follows by the assumptions on $f$ that the Schur
complement of the last entry in $f[U(A)]$ is positive semidefinite, i.e.,
\[ f[U(A)] = \begin{pmatrix} f[a^{-1} {\bf u } {\bf u}^T] & f[{\bf u}] \\
f[{\bf u }]^T & f(a) \end{pmatrix} \geq 0 \quad \implies \quad f[a^{-1}
{\bf u}{\bf u}^T] \geq \frac{f[{\bf u}] f[{\bf u}]^T}{f(a)}. \]

Furthermore, $f(0) = 0$ and $l = \deg_{G_k} v_k \leq d$. Hence $f[M+N]
\geq f[M] + f[N]$ for all $M \in \bp_l(\R)$ and $N \in \bp_l^1(\R)$. Set
$N := a^{-1} {\bf u} {\bf u}^T$ and $M := P-N$, and compute using the
above analysis:
\[ f[P] - f(a)^{-1} f[{\bf u}] f[{\bf u}]^T - f[P - a^{-1} {\bf u} {\bf
u}^T] \geq f[P] - f[a^{-1} {\bf u} {\bf u}^T] - f[P - a^{-1} {\bf u} {\bf
u}^T] \geq 0. \]

\noindent It follows that $S_{f[A]} \geq f[S_A] \geq 0$, whence $f[A] \in
\bp_{G_{k+1}}(\R)$ as claimed. This completes the induction step. The
result now follows by setting $k=n$.
\end{proof} 

We now study how the set of powers preserving positivity on $\bp_G$ can
be related to the corresponding set of powers for $\bp_{G/v}$, for
arbitrary graphs $G$.

As an illustration of Theorem \ref{Tmain}, we compute in Corollary
\ref{Cexplicit} the critical exponents of well-known chordal graphs
explicitly. Recall that an \textit{Apollonian graph} is a planar graph
formed from a triangle graph by iteratively adding an interior point as
vertex, and connecting it to all three vertices of the smallest triangle
subgraph in whose interior it lies. A graph is \textit{outerplanar} if
every vertex of the graph lies in the unbounded face of the graph in a
planar drawing. An outerplanar graph is \textit{maximal} if adding an
edge makes it non-outerplanar.
A graph $G = (V,E)$ is \textit{split} if its vertices can be partitioned
into a clique $C$ and an independent subset $V \setminus C$.
Finally, the \textit{band graph} with $n$ vertices
$\{1,\dots,n\}$ and bandwidth $d$ is the graph where $(i,j) \in E$ if and
only if $i \ne j$ and $|i-j| \leq d$. For references, see
e.g.~\cite{chordal-split,Dym-Gohberg,golumbic,Max-Outerplanar}.

\begin{corollary}\label{Cexplicit}
The critical exponents of some important chordal graphs are given in
Table \ref{Table_chordal}.
\end{corollary}

\begin{table}[h]
\begin{tabular}{|c|c|}
\hline
Graph $G$ & $CE_H(G), CE_H^\psi(G), CE_H^\phi(G)$ \\ \hline
Tree & 1 \\
Complete graph $K_n$ & $n-2$ \\
Minimal planar triangulation of $C_n$ for $n \geq 4$ & 2 \\
Apollonian graph, $n \geq 3$ & $\min(3,n-2)$ \\
Maximal outerplanar graph, $n \geq 3$ & $\min(2,n-2)$\\
Band graph with bandwidth $d \leq n$ & $\min(d,n-2)$\\
Split graph with maximal clique $C$ & $\max(|C|-2, \max \deg(V \setminus
C))$\\ \hline
\end{tabular}
\medskip
\caption{Critical exponents of important families of chordal graphs with
$n$ vertices.}
\label{Table_chordal}
\end{table}

\begin{proof}
We will prove the result only for band graphs. First, if $n=d,d+1$, then
$G = K_d, K_{d+1}^{(1)}$ respectively, and so the critical exponents are
$d-2$ and $d-1$, which shows the result. Now suppose $n \geq d+2$. The
maximal cliques of $G$ are 
\begin{equation*}
C_l := \{l, l+1, \dots, l+d\} \qquad 1 \leq l \leq n-d. 
\end{equation*}

\noindent It is not difficult to verify that this enumeration of the
maximal cliques of $G$ is perfect. The largest clique has size $d+1$ and
the largest separator (as defined in \eqref{EhistoriesC}) has size $d$.
It follows from Theorem \ref{Tmain} (see Remark \ref{Ralternative_crit})
that the three critical exponents of $G$ are equal to $d$. 
\end{proof}

\begin{remark}
Another important family of chordal graphs that is widely used in
applications is the family of \textit{interval graphs} \cite[Chapter
8]{golumbic}. Given a family $V$ of intervals in the real line, the
corresponding interval graph has vertex set equal to $V$, and two
vertices are adjacent if the corresponding intervals intersect. Interval
graphs are known to be chordal; moreover, to compute their critical
exponents we define the \textit{height function} at $x \in \R$ to be the
number of intervals containing $x$. It is standard that the maximal
cliques correspond precisely to the intervals containing the local maxima
of the height function; see e.g.~\cite[Section 2]{IntervalG}. The
critical exponent of interval graphs can then be easily computed by using
Corollary \ref{Cformula}. 
\end{remark}

Note that every non-chordal graph $G$ is contained in a minimal
triangulation $G_\Delta$. This triangulation immediately provides an
upper bound on the critical exponents for preserving positivity for $G$.
A lower bound is provided by $r-2$, where $r$ is the size of the largest
clique in $G$. The critical exponents of some non-chordal graphs are
studied in more detail in Section \ref{Snon-chordal}. 

\section{Non-chordal graphs}\label{Snon-chordal}

In the remainder of the paper, we discuss power functions preserving
positivity on $\bp_G$ for graphs $G$ that are not chordal. We begin by
extending Corollary \ref{Tchordal_gen} to general graphs. Recall that a
\textit{decomposition} of a graph $G = (V,E)$ is a partition $(A,C,B)$ of
$V$, where $C$ separates $A$ from $B$ (i.e., every path from a vertex $a
\in A$ to a vertex $b \in B$ contains a vertex in $C$), and $G_C$ is
complete. A graph is said to be \textit{prime} if it admits no such
decomposition. For example, every cycle is prime. A decomposition
separates a graph into two components $G_{A \cup C}$ and $G_{B \cup C}$.
Iterating this process until it cannot be performed anymore produces
\textit{prime components} of the graph $G$. The resulting prime
components can be ordered to form a perfect sequence (as defined after
\eqref{Ehistories}) - see \cite{Diestel1989, Roverato2002}. When $G$ is
chordal, its prime components are all complete. Conversely, if the prime
components of a graph are complete, the graph is chordal by
\cite{Dirac1961} (see also \cite[Theorem 3.1]{Diestel1989}).

\begin{theorem}\label{Tprime}
Let $G$ be a graph with a perfect ordering $\{B_1, \dots, B_k\}$ of its
prime components, and let $f : \R \to \R$ be such that $f(0) = 0$. Define
\begin{equation*}
s := \max_{i=1, \dots, k} |S_i|, 
\end{equation*}

\noindent where $S_i$ is defined as in \eqref{Ehistories}. If $f[-]$
preserves positivity on $\bp_{B_i}$ for all $1 \leq i \leq k$ and is
Loewner super-additive on $\bp_{K_s}$, then $f[-]$ preserves positivity
on $\bp_G$.
\end{theorem}

\begin{proof}
The proof is the same as the proof of Theorem \ref{Tmain}. 
\end{proof}

As a consequence, we immediately obtain the following corollary. 

\begin{corollary}
In the notation of Theorem \ref{Tprime}, let $\alpha > 0$ and let $f =
\psi_\alpha$ or $f = \phi_\alpha$. Suppose $f[-]$ preserves positivity on
$\bp_{B_i}(\R)$ for every prime component $B_i$ of $G$, and $\alpha \geq
|S_i|$ for all $i$. Then $f$ preserves Loewner positivity on $\bp_G(\R)$.
\end{corollary}

A natural question of interest is thus to determine the critical
exponents of prime graphs, and other simple non-chordal graphs. In the
next two subsections, we examine the case of cycles and of bipartite
graphs. Along the way, we develop general techniques to compute critical
exponents of graphs, including studying the Schur complement of a graph,
and appending paths to graphs. We conclude the paper by constructing many
more examples of non-chordal graphs for which the critical exponent can
be obtained explicitly by forming coalescences of graphs.

\subsection{Cycles, Schur complements, and path addition}

We begin by proving that for even cycles, the critical exponents
$CE_H(G), CE_H^\psi(G), CE_H^\phi(G)$ are \textit{not} all equal, which
is unlike the case of chordal graphs.

\begin{proposition}\label{Pcycle}
For all $n \geq 3$, 
\begin{equation*}
\calh_{C_n} = \calh_{C_n}^\psi = [1,\infty), \textrm{ and }
\calh^\phi_{C_4} = [2,\infty).
\end{equation*}
Moreover, for $n>4$, $[2,\infty)
\subset \calh_{C_n}^\phi \subset [1,\infty)$, with $1 \notin
\calh_{C_n}^\phi$ for $n$ even.
\end{proposition}

We prove Proposition \ref{Pcycle} in this section. Along the way, we
describe various constructions on a graph under which the Hadamard
critical exponents can be controlled. The first of these constructions is
termed the Schur complement, and generalizes the pendant edge
construction in Theorem \ref{Taddtrees}, which shows that the
$\calh$-sets do not change when a tree is pasted on a vertex of a graph. 

\begin{definition}
Let $G = (V,E)$ be a graph and let $v \in V$. Define the \emph{Schur
complement graph of $G$ with respect to $v$}, denoted $G / v$, to be the
simple graph $G/v := (V \setminus \{v\}, E')$, where $(i,j) \in E'$ if
and only if one of the following condition holds:
\begin{enumerate}
\item  $(i,j) \in E \cap \left(V \setminus \{v\}\right) \times \left(V
\setminus \{v\}\right)$;
\item $(i,v) \in E$ and $(j,v) \in E$. 
\end{enumerate} 
\end{definition}

\noindent For instance, the Schur complement of any vertex in a path
$P_n$ for $n>2$, cycle $C_n$ for $n>3$, or complete graph $K_n$ for
$n>2$, is $P_{n-1}, C_{n-1}, K_{n-1}$ respectively. We remark that this
operation has also been referred to as ``orthogonal removal'' in the
context of the minimum positive semidefinite rank of a graph.

The definition of the graph Schur complement is designed to be compatible
with the Schur complement of a matrix in $\bp_G$. Namely, if we take the
Schur complement of $A \in \bp_G$ with respect to its $(v,v)$-th entry
which is positive, then the resulting matrix is in $\bp_{G /v}$. This
makes the construction a very relevant one, as Schur complements provide
a crucial tool for computing Hadamard critical exponents. For example,
the following result relates $\calh^\psi_G$ and $\calh^\psi_{G/v}$ for
vertices $v \in V(G)$ with independent neighbors.

\begin{theorem}\label{Tindep}
Suppose $G = (V,E)$ is not a disjoint union of copies of $K_2$, and $v
\in V$. Then $1 + \calh_{G / v} \subset \calh_G$. Suppose $v$ has $k>1$
neighbors in $G$, and they are independent. Then $\calh_{G / v}^\psi
\subset \calh^\psi_G \subset \calh_G$.
\end{theorem}

\begin{proof}
Without loss of generality, assume $V = \{1,\dots, m\}$, $v = m$, and
$a_{mm} \not= 0$. Let $\alpha \in \calh_{G/v}$ and $A = (a_{ij}) \in
\bp_G([0,\infty))$. If $\zeta:= (a_{1m}, a_{2m}, \dots,
a_{mm})^T/\sqrt{a_{mm}}$, then, as in \cite[Equation (2.1)]{FitzHorn},
\begin{equation}
A^{\circ (\alpha + 1)} = \zeta^{\circ (\alpha + 1)} (\zeta^{\circ (\alpha
+ 1)})^T + (\alpha + 1)\int_0^1 (A-\zeta \zeta^T) \circ (tA + (1-t) \zeta
\zeta^T)^{\circ \alpha}\ dt. 
\end{equation}

\noindent By \cite[Lemma 2.1]{FitzHorn}, $A-\zeta\zeta^T$ is positive
semidefinite and its last row and column vanish. Moreover, the principal
submatrix obtained by taking the first $m-1$ rows and columns of $tA +
(1-t) \zeta \zeta^T$ belongs to $\bp_{G / \{v\}}([0,\infty))$. The first
assertion now follows immediately from the hypothesis and the Schur
product theorem.

To show the second assertion for $\psi_\alpha$, assume without loss of
generality that $v$ has independent neighbors $v_1, \dots, v_k$ with
$k>1$. Suppose now that $G$ has $m = n+k+1$ vertices, with $v_i = n+i$
for $1 \leq i \leq k$ and $v = n+k+1$.
Now since the induced subgraph on vertices $v_1,v_2,v$ is $P_3$, we have
$\calh_G \subset [1,\infty)$ by Theorem \ref{Taddtrees}.
Thus, suppose $1 \leq \alpha \in \calh^\psi_{G / v}$, and $A \in \bp_G$
is of the form
\begin{equation}\label{Eindep}
A := \begin{pmatrix}
B_{n \times n} & {\bf U}_{n \times k} & {\bf 0}_{n \times 1}\\
{\bf U}^T & D_{k \times k} & {\bf a}_{k \times 1}\\
{\bf 0}^T & {\bf a}^T & p_{k+1}
\end{pmatrix},
\end{equation}

\noindent for suitable ${\bf a}$ and ${\bf U}$, with $D$ the diagonal
matrix ${\rm diag}(p_1, \dots, p_k)$. If $p_{k+1} = 0$ then $a_i = 0\
\forall i$ and hence $\psi_\alpha[A] \in \bp_G$ as desired. If instead
$p_{k+1} > 0$ then $\psi_\alpha[A] \in \bp_G$ if and only if the Schur
complement of $p_{k+1}^\alpha$ in $\psi_\alpha[A]$ is in $\bp_{G/v}$.
This Schur complement equals 
\[ S_{\psi_\alpha[A]} := \begin{pmatrix} \psi_\alpha[B] &
\psi_\alpha[{\bf U}]\\ \psi_\alpha[{\bf U}]^T & D^{\circ \alpha} -
p_{k+1}^{-\alpha} \psi_\alpha[{\bf a}] \psi_\alpha[{\bf a}]^T
\end{pmatrix}. \]

\noindent Since $\alpha \in \calh^\psi_{G / v}$, hence $\alpha \in
\calh^\psi_{K_k}$, whence $\psi_\alpha[D - p_{k+1}^{-1} {\bf a} {\bf
a}^T] \in \bp_k$. We now claim that
\[ D^{\circ \alpha} - p_{k+1}^{-\alpha} \psi_\alpha[{\bf a}]
\psi_\alpha[{\bf a}]^T - \psi_\alpha[D - p_{k+1}^{-1} {\bf a} {\bf a}^T]
\in \bp_k. \]

\noindent Indeed, this difference is a diagonal matrix with diagonal
entries
\[ p_i^\alpha - \frac{a_i^{2\alpha}}{p_{k+1}^\alpha} - \left( p_i -
\frac{a_i^2}{p_{k+1}} \right)^\alpha, \]

\noindent and these are nonnegative since $\alpha \geq 1$. The claim thus
follows, and in turn implies that $S_{\psi_\alpha[A]} - \psi_\alpha[S_A]
\in \bp_{n+k}$, where $S_A \in \bp_{G / v}$ is the Schur complement of
$p_{k+1}$ in $A$. Since $\alpha \in \calh^\psi_{G/v}$, hence
$S_{\psi_\alpha[A]} \in \bp_{G / v}$ as desired. We conclude that
$\psi_\alpha[A] \in \bp_G$, and so $\alpha \in \calh^\psi_G$ as claimed.
\end{proof}

As a consequence of Theorem \ref{Tindep}, we now discuss a construction
starting from a graph $G$ and connecting two non-adjacent vertices in $G$
by a path.

\begin{definition}
Fix a graph $G = (V,E)$, vertices $v_1, v_2 \in V$, and $m \in \N$. If
$v_1,v_2$ are adjacent and $m=1$ then we set $G_1(v_1,v_2) := G$.
Otherwise define $G_m(v_1,v_2)$ to be the graph $G$, together with an
additional path of edge-length $m$ connecting $v_1, v_2$.
\end{definition}

The following useful result is a consequence of Theorem
\ref{Tindep}.

\begin{corollary}\label{Cbarhat}
Suppose $G = (V,E)$ is not a disjoint union of copies of $K_2$, and $v_1,
v_2 \in V$.
Then $\calh_{G_2(v_1,v_2)}^\psi \subset \calh_{G_3(v_1,v_2)}^\psi \subset
\calh_{G_4(v_1,v_2)}^\psi \subset \cdots \subset \calh_G^\psi$; moreover,
$\calh_{G_1(v_1,v_2)}^\psi \subset \calh_{G_2(v_1,v_2)}^\psi$ if $v_1,
v_2$ are not adjacent.
\end{corollary}

\begin{proof}
First note that $\calh^\psi_{G_m(v_1,v_2)} \subset \calh^\psi_G$ since $G
\subset G_m(v_1,v_2)$ for all $m \geq 1$. We first show that
$\calh_{G_1(v_1,v_2)}^\psi \subset \calh_{G_2(v_1,v_2)}^\psi$ if $v_1,
v_2$ are not adjacent in $G$. Suppose $G$ has $n+2$ vertices, with $v_i =
n+i$ for $i=1,2$. Let the additional vertex in $G_2(v_1,v_2)$ be $v =
n+3$. Then by Theorem \ref{Tindep},
\begin{equation}\label{Eincl}
\calh_{G_1(v_1,v_2)}^\psi = \calh_{G_2(v_1,v_2) / v}^\psi \subset
\calh_{G_2(v_1,v_2)}^\psi.
\end{equation}

\noindent Now fix $m \in \N$ and let $G'_m(v_1)$ be the graph obtained by
attaching a path of edge-length $m \in \N$ at one end to $v_1$, and
leaving the other end free/pendant. Let $w_m$ be the free vertex with
$w_0 := v_1$; then $G_{m+1}(v_1,v_2) = G_1(G'_m(v_1),w_m,v_2)$. Hence by
\eqref{Eincl},
\[ \calh_{G_{m+1}(v_1,v_2)}^\psi = \calh_{G_1(G'_m(v_1),w_m,v_2)}^\psi
\subset \calh_{G_2(G'_m(v_1),w_m,v_2)}^\psi =
\calh_{G_{m+2}(v_1,v_2)}^\psi. \]
\end{proof}

It is now possible to obtain information about the critical exponents for
cycle graphs.

\begin{proof}[Proof of Proposition \ref{Pcycle}]
For $n=3$ the result is clear from Theorem \ref{Tcomplete}, since $C_3 =
K_3$. We compute for $n \geq 4$ and any vertex $v_m \in C_m$ (for $m \leq
n$), using Theorems \ref{Tcomplete}, \ref{Taddtrees}, and \ref{Tindep}:
\[ [1,\infty) = \calh_{P_3} \supset \calh_{C_n} \supset \calh_{C_n}^\psi
\supset \calh_{C_n/v_n}^\psi = \calh_{C_{n-1}}^\psi \supset \cdots
\supset \calh_{C_3}^\psi = [1,\infty). \]

\noindent It follows that $\calh_{C_n} = \calh_{C_n}^\psi = [1,\infty)$
for $n \geq 3$.

We now compute $\calh_{C_4}^\phi$. Note that the matrix $A :=
(\cos((j-k)\pi/4))_{j,k=1}^4$, which was well-studied in
\cite{Bhatia-Elsner,GKR-crit-2sided}, lies in $\bp_{C_4}(\R)$. It was
further shown in \cite{Bhatia-Elsner} that $\phi_\alpha[A] \notin
\bp_4(\R)$ for $\alpha \in (0,2)$. Thus $\calh_{C_4}^\phi \subset
[2,\infty)$. On the other hand, $\calh_{C_4}^\phi \supset
\calh_{K_4}^\phi = [2,\infty)$ by Theorem \ref{Tcomplete}, which shows
the result for $C_4$. Next for general $n$, $\calh_{C_n}^\phi \subset
[1,\infty)$ by Theorem \ref{Taddtrees} since $C_n \supset P_3$. On the
other hand, observe that
\begin{equation}\label{Ehcec}
|CE_H^\psi(G) - CE_H^\phi(G)| \leq 1
\end{equation}

\noindent for all graphs $G$, because $\psi_\alpha(x) = x
\phi_{\alpha-1}(x)$ and $\phi_\alpha(x) = x \psi_{\alpha-1}(x)$ for
$\alpha \in \R$, so that $1 + \calh_G^\phi \subset \calh_G^\psi \subset
\calh_G$ by the Schur product theorem, and similarly for $\calh_G^\phi$.
Applying \eqref{Ehcec} for $G = C_n$, it follows that $[2,\infty) \subset
\calh_{C_n}^\phi$ by the Schur product theorem, since $\phi_{\alpha +
1}(x) = x \cdot \psi_\alpha(x)$ and $[1,\infty) \subset
\calh_{C_n}^{\psi}$. Finally, observe using \cite[Example
5.2]{drton_yu_2010} that $1 \not\in \calh^\phi_{C_{2n}}$ for $n \geq 2$,
whence $CE_H^\phi(C_{2n}) > 1$.
\end{proof}

\begin{remark}
The same analysis as above leads us to conclude that if $G_n$ is the
non-chordal graph on $2n$ vertices with only the
``diameter'' edges $(1, n+1), \dots, (n,2n)$ missing, then
\begin{equation}\label{Edense}
\calh^\phi_{G_n} = 2\N \cup [2n-2,\infty).
\end{equation}
This assertion is proved using the properties of the matrix $A_n :=
(\cos((j-k)\pi/(2n)))_{j,k=1}^{2n} \in \bp_{G_n}$, which were explored in
\cite{Bhatia-Elsner,GKR-crit-2sided}. In particular,
it follows from \eqref{Edense} that $\calh_{G_n} = \N \cup
[2n-2,\infty)$ and $CE_H^\psi(G_n) \in [2n-3,2n-2]$.
\end{remark}

\noindent Proposition \ref{Pcycle} allows us to strengthen Corollary
\ref{Cbarhat} in the particular case where $G_0 = K_4^{(1)}$ or $K_4$. 

\begin{proposition}\label{Pdecomp}
Suppose $H_0 = K_4^{(1)}$ or $K_4$. Now given a graph $H_m$ for $m \geq
0$ and an integer $n_{m+1} \geq 3$, create a new graph $H_{m+1}$ by
attaching a cycle $C_{n_{m+1}}$ to $H_m$ along any common edge. Then for
all $m \geq 0$, 
\begin{equation*}
\calh_{G_m} = \N \cup [2,\infty), \quad \calh_{G_m}^{\psi} = (-1+2\N)
\cup [2,\infty), \quad \calh_{G_m}^{\phi} = 2\N \cup [2,\infty). 
\end{equation*}
In particular, $CE_H(H_m) = CE_H^\psi(H_m) = CE_H^\phi(H_m) = 2$ for all
$m \geq 0$.
\end{proposition}

\begin{proof}
First note using Theorems \ref{Tcomplete} and \ref{Tmain} that
$\calh_{H_m} \setminus \{ 1 \}, \calh_{H_m}^\psi \setminus \{ 1 \},
\calh_{H_m}^\phi \subset [2,\infty)$ for all $m \geq 0$. To show that
$[2,\infty)$ is contained in the three $\calh$-sets we use induction on
$m \geq 0$. The result clearly holds for $H_0$ by Theorems
\ref{Tcomplete} and \ref{Tmain}. Now assume the result holds for $H_m$,
and suppose $C_{n_{m+1}}$ is attached to $H_m$ along the common edge
$(1,2)$ (without loss of generality). Let $A := V(H_m) \setminus
\{1,2\}$, $C := \{1,2\}$, and $B := V(C_{n_{m+1}}) \setminus \{1,2\}$,
where $V(H_m), V(C_{n_{m+1}})$ denote the vertex sets of $H_m$ and
$C_{n_{m+1}}$ respectively. For every $\alpha \geq 2$, the maps
$\psi_\alpha, \phi_\alpha$ preserve positivity on $\bp_{H_m}(\R)$ by the
induction hypothesis, and on $\bp_{C_{n_{m+1}} }(\R)$ by Proposition
\ref{Pcycle}.
Moreover, $\psi_\alpha, \phi_\alpha$ are Loewner super-additive on
$\bp_S(\R)$ by Theorem \ref{Tsuperadd}. Hence $\calh^\psi_{H_{m+1}} = \{
1 \} \cup [2,\infty)$ and $\calh^\phi_{H_{m+1}} = [2,\infty)$. Finally,
since $H_0 \subset H_{m+1}$, this implies: $\N \cup \calh^\phi_{H_{m+1}}
\subset \calh_{H_{m+1}} \subset \calh_{H_0} = \{ 1 \} \cup [2,\infty)$.
This shows the assertion for $\calh_{H_{m+1}}$, and the proof is
complete.
\end{proof}

\subsection{Bipartite graphs}

Another commonly encountered family of non-chordal graphs are the
bipartite graphs. We now examine the critical exponents of these graphs. 

\begin{theorem}\label{Tbipartite}
Suppose $G$ is a connected bipartite graph with at least $3$ vertices.
Then,
\[ \calh_G = [1,\infty), \qquad [2,\infty) \subset
\calh^\phi_G \subset [1,\infty), \qquad \{ 1 \} \cup [3,\infty) \subset
\calh^\psi_G \subset [1,\infty). \]

\noindent If moreover $K_{2,2} \subset G \subset K_{2,m}$ for some $m
\geq 2$, then $\calh^\phi_G = [2,\infty)$ and $\{ 1 \} \cup [2,\infty)
\subset \calh^\psi_G \subset [1,\infty)$.
\end{theorem}


Theorem \ref{Tbipartite} has a very surprising conclusion: it shows that
broad families of dense graphs such as complete bipartite graphs have
small critical exponents that do not grow with the number of vertices of
the graphs. As a consequence, small entrywise powers of a positive
semidefinite matrix with such a structure of zeros preserves positivity.
This is important since such procedures are often used to regularize
positive definite matrices (e.g.~ covariance/correlation matrices), where
the goal is to minimally modify the entries of the original matrix. Note
that such a result is in sharp contrast to the general case where there
is no underlying structure of zeros.

\begin{proof}
\noindent \textbf{Step 1: Complete bipartite graphs.}
We begin by proving that the complete bipartite graph $K_{n,n}$
satisfies: $\calh_{K_{n,n}} = [1,\infty)$ for all $n \geq 2$.
Indeed, $P_3 \subset K_{n,n}$ since $n \geq 2$, so we conclude via
Theorem \ref{Taddtrees} that $\calh_{K_{n,n}} \subset \calh_{P_3} =
[1,\infty)$. To show the reverse inclusion, let $\alpha > 0$,
$m,n \in \N$, and let 
\begin{equation*}
A = \begin{pmatrix} D_{m \times m}
& X_{m \times n}\\ X^T & D'_{n \times n} \end{pmatrix} \in
\bp_{K_{m,n}}([0,\infty)), 
\end{equation*}
with $\max(m,n) > 1$, and where $D, D'$ are
diagonal matrices. Given $\epsilon > 0$, define the matrix
\[ X_{D,D'}(\epsilon,\alpha) := (D + \epsilon \Id_m)^{\circ (-\alpha/2)}
\cdot X^{\circ \alpha} \cdot (D' + \epsilon \Id_n)^{\circ (-\alpha/2)}.
\]

Also observe that for all block diagonal matrices $A$ of the above form
and all $\epsilon, \alpha > 0$,
\[ (A + \epsilon \Id_{m+n})^{\circ \alpha} = {\bf D}_\epsilon
\begin{pmatrix} \Id_m & X_{D,D'}(\epsilon,\alpha)\\
X_{D,D'}(\epsilon,\alpha)^T & \Id_n \end{pmatrix} {\bf D}_\epsilon, \]

\noindent where 
\begin{equation*}
{\bf D}_\epsilon := \begin{pmatrix} (D +
\epsilon \Id_m)^{\circ \alpha/2} & {\bf 0}\\ {\bf 0} & (D' + \epsilon
\Id_n)^{\circ \alpha/2} \end{pmatrix}.
\end{equation*}
We now compute for $\alpha,
\epsilon > 0$:
\begin{align*}
&\ (A + \epsilon \Id_{m+n})^{\circ \alpha} \in
\bp_{K_{m,n}}([0,\infty))\\
\Longleftrightarrow &\ \begin{pmatrix} \Id_m &
X_{D,D'}(\epsilon,\alpha)\\ X_{D,D'}(\epsilon,\alpha)^T & \Id_n
\end{pmatrix} \in \bp_{K_{m,n}}([0,\infty))\\
\Longleftrightarrow &\ \Id_m - X_{D,D'}(\epsilon,\alpha)
X_{D,D'}(\epsilon,\alpha)^T \in \bp_m(\R)\\
\Longleftrightarrow &\ \| u \| \geq \| X_{D,D'}(\epsilon,\alpha)^T u \|,
\ \forall u \in \R^n\\
\Longleftrightarrow &\ \sigma_{\max}(X_{D,D'}(\epsilon,\alpha)) \leq 1,
\end{align*}

\noindent where $\sigma_{\max}$ denotes the largest singular value. Now
note that if $m=n$, then the above calculation shows that $(A + \epsilon
\Id_{2n})^{\circ \alpha} \in \bp_{K_{n,n}}([0,\infty))$ if and only if
$\rho(X_{D,D'}(\epsilon,\alpha)) \leq 1$, where $\rho$ denotes the
spectral radius.

To finish this first step of the proof, now suppose $\alpha \geq 1$ and
$A \in \bp_{K_{n,n}}([0,\infty))$. Then $A + \epsilon \Id \in
\bp_{K_{n,n}}([0,\infty))$ for all $0 < \epsilon \ll 1$, so by the above
analysis with $\alpha = 1$, $\rho(X_{D,D'}(\epsilon,1)) \leq 1$ for all
$0 < \epsilon \ll 1$. Applying \cite[Lemma
5.7.8]{Horn_and_Johnson_Topics} implies that
\begin{equation*} 
\rho(X_{D,D'}(\epsilon,\alpha)) \leq \rho(X_{D,D'}(\epsilon,1))^\alpha
\leq 1. 
\end{equation*}

\noindent It follows from the above analysis and the continuity of
entrywise powers that $A^{\circ \alpha} \in \bp_{K_{n,n}}([0,\infty))$.
Thus $[1,\infty) \subset \calh_{K_{n,n}}$.
\medskip

\noindent \textbf{Step 2: General bipartite graphs.}
We now prove the result for a general bipartite graph. Suppose $G =
(V,E)$ is any connected bipartite graph on $m,n$ vertices, with $m+n =
|V| \geq 3$ and $n \geq m$. Then by the previous step,
\[ P_3 = K_{2,1} \subset G \subset K_{n,n} \quad \implies \quad
[1,\infty) \subset \calh_G \subset \calh_{K_{n,n}} = [1,\infty), \]

\noindent which shows that $\calh_G = [1,\infty)$. Next, suppose $\alpha
\geq 2$ and $A \in \bp_G(\R)$. Then $A^{\circ 2} = A \circ A \in
\bp_G([0,\infty))$ by the Schur product theorem, so by the previous
assertion, $\phi_\alpha[A] = (A \circ A)^{\circ \alpha/2} \in
\bp_G([0,\infty))$. It follows immediately that $[2,\infty)
\subset \calh^\phi_G \subset [1,\infty)$. In turn, this implies via
\eqref{Ehcec} that $\{ 1 \} \cup [3,\infty) \subset \calh^\psi_G \subset
[1,\infty)$.\medskip

To conclude the proof, suppose further that $C_4 = K_{2,2} \subset G
\subset K_{2,m}$. To study $\calh_G$ we will use the family of split
graphs $K_{K_2,m}$ for $m \geq 2$. These are chordal graphs with $m+2$
vertices, with vertices $m+1, m+2$ connected to every other vertex (and
to each other). By Theorem \ref{Tmain}, Proposition \ref{Pcycle}, and the
definition of $G$, we obtain
\begin{equation*}
\{ 1 \} \cup [2,\infty) = \calh^\psi_{K_{K_2,m}} \subset \calh^\psi_G
\subset \calh^\psi_{C_4} = [1,\infty), \qquad
[2,\infty) = \calh^\phi_{K_{K_2,m}} \subset \calh^\phi_G \subset
\calh^\phi_{C_4} = [2,\infty).
\end{equation*}

\noindent This concludes the proof.
\end{proof}

\subsection{Coalescences}

In this concluding section, we show how many more examples of non-chordal
graphs can be constructed by forming coalescence of graphs. Recall that
the \textit{coalescences} of two graphs $G_1, G_2$ is the graph obtained
from their disjoint union $G_1 \coprod G_2$ by identifying a vertex from
both of them \cite{Grone-Merris,coalescence}. We now discuss how to
extend the proof-strategy of Theorem \ref{Tchordal} to such graphs.

\begin{proposition}[Coalescence graphs]\label{Pcoalescence}
Suppose $G_1, \dots, G_k$ are connected graphs with at least one edge
each, and $G$ is any coalescence of $G_1, \dots, G_k$ for some $k>1$.
Also suppose $f : \R \to \R$ with $f(0) = 0$. Then $f[-]$ preserves
positivity on $\bp_G(\R)$ if and only if: 
\begin{enumerate}
\item $f[-]$ preserves positivity on each $\bp_{G_i}(\R)$, and 
\item $f$ is continuous and super-additive on $[0,\infty)$.
\end{enumerate}
In particular for any $\alpha \in \R$, the power function $\psi_\alpha$
or $\phi_\alpha$ preserves positivity on $\bp_G(\R)$ if and only if it
does so on $\bp_{G_i}(\R)$ for all $1 \leq i \leq k$ and $\alpha \geq 1$.
In other words,
\[ \calh^\psi_G = [1,\infty) \cap \bigcap_{i=1}^k \calh^\psi_{G_i},
\qquad \calh^\phi_G = [1,\infty) \cap \bigcap_{i=1}^k \calh^\phi_{G_i}.
\]
\end{proposition}

\begin{remark}\hfill
\begin{enumerate}
\item Note that the characterization provided by Proposition
\ref{Pcoalescence} is independent of which nodes in the graphs $G_i$ are
identified with one another. Moreover, the critical exponents satisfy:
$CE(G) = \max(1, \max_i CE(G_i))$.

\item Also observe that when $k=2$ and $G_2 = K_2$, the resulting graph
$G$ in Proposition \ref{Pcoalescence} is the graph $G_1$ with one pendant
edge added. Proposition \ref{Pcoalescence} therefore implies the
conclusion of Theorem \ref{Taddtrees}.
\end{enumerate}
\end{remark}

The proof of Proposition \ref{Pcoalescence} relies on a stronger result
that is akin to Theorem \ref{Tchordal}, but holds for arbitrary graphs:

\begin{theorem}\label{Tseparator}
Let $G = (V,E)$ be a nonempty graph and let $(A,C,B)$ be a partition of
$V$ where $C$ separates $A$ from $B$. Also let $f: \R \to \R$ be such
that $f(0) = 0$. Suppose $f[-]$ preserves positivity on $\bp_G$. Then 
\begin{enumerate}
\item $f[-]$ preserves positivity on $\bp_{A \cup C}$ and on $\bp_{B \cup
C}$;
\item $f$ is continuous on $[0,\infty)$, and 
\item $f[-]$ is Loewner super-additive on $\bp_m(\R)$, whenever there
exist $A' \subset A, C' \subset C, B' \subset B$ such that
$G_{A'},G_{C'},G_{B'}$ are cliques of size $m$, and every vertex of $A'$
and $B'$ is connected to every vertex in $C'$.
\end{enumerate}
\end{theorem}

\noindent Note that the converse to Theorem \ref{Tseparator} was proved
in Theorem \ref{Tchordal}(1).

\begin{proof}
Clearly, $f[-]$ preserves positivity on $\bp_{A \cup C}$ and on $\bp_{B
\cup C}$ since they are induced subgraphs of $G$. Also $f$ preserves
positivity on $\bp_{K_2}(\R) = \bp_2(\R)$, whence $f$ is continuous on
$(0,\infty)$ by \cite[Theorem 1.2]{Horn}. Moreover, $f$ is
right-continuous at $0$ as shown at the beginning of the proof of
\cite[Theorem C]{Guillot_Khare_Rajaratnam2012}. Now, write the vertices
of $G$ in the order $A', C', B', V \setminus (A' \cup C' \cup B')$, and
consider the matrices
\[ \calm_1(N) := \begin{pmatrix}
N & N & {\bf 0} & {\bf 0}\\
N & N & {\bf 0} & {\bf 0}\\ 
{\bf 0} & {\bf 0} & {\bf 0} & {\bf 0}\\ 
{\bf 0} & {\bf 0} & {\bf 0} & {\bf 0}
\end{pmatrix}, \qquad
\calm_2(N) := \begin{pmatrix}
{\bf 0} & {\bf 0} & {\bf 0} & {\bf 0}\\
{\bf 0} & N & N & {\bf 0}\\
{\bf 0} & N & N & {\bf 0}\\
{\bf 0} & {\bf 0} & {\bf 0} & {\bf 0}
\end{pmatrix}, \qquad N \in \bp_m(\R). \]

\noindent Clearly $\calm_1(N) \in \bp_{G_{A \cup C}}$ and $\calm_2(N)\in
\bp_{G_{B \cup C}}$ for all $N \in \bp_m(\R)$. Thus given any $N_1, N_2
\in \bp_m(\R)$ and $\epsilon > 0$, it follows that
$f[\calm_1(N_1) + \calm_2(N_2)] + \calm_1[\epsilon \Id_m] +
\calm_2[\epsilon \Id_m] \in \bp_G(\R)$, i.e.,
\begin{equation*}
\begin{pmatrix} f[N_1] + \epsilon \Id_m & f[N_1] + \epsilon \Id_m & {\bf
0}_{m \times m}\\ f[N_1] + \epsilon \Id_m & f[N_1 + N_2] + 2 \epsilon
\Id_m & f[N_2] + \epsilon \Id_m\\ {\bf 0}_{m \times m} & f[N_2] +
\epsilon \Id_m & f[N_2] + \epsilon \Id_m \end{pmatrix} \in \bp_G(\R). 
\end{equation*}

\noindent Proceeding as in the proof of Theorem \ref{Tchordal} (see
\eqref{Esuper_rank1}), it follows that $f[N_1 + N_2] - f[N_1] - f[N_2]
\in \bp_m(\R)$, i.e., $f[-]$ is Loewner super-additive on $\bp_m(\R)$.
\end{proof}

Having proved Theorem \ref{Tseparator}, it is now possible to prove
Proposition \ref{Pcoalescence} about coalescences of graphs.

\begin{proof}[Proof of Proposition \ref{Pcoalescence}]
We prove the result by induction on $k$, with the base case of $i=2$ and
the higher cases proved similarly. Let $G'_i$ denote the coalescence of
the graphs $G_1, \dots, G_i$, for each $i = 1, \dots, k$. Applying
Theorem \ref{Tseparator} with $|C| = 1$ corresponding to the vertex along
which $G'_{i-1}$ and $G_i$ are coalesced, we conclude that $f[-]$
preserves positivity on $\bp_{G_i}$ for all $i$, $f$ is continuous on
$[0,\infty)$, and $f$ is super-additive on $[0,\infty)$. Similarly, if
$f[-]$ preserves positivity on $\bp_{G_i}$ for all $i$, and is
super-additive on $[0,\infty)$, then $f[-]$ preserves positivity on
$\bp_G$ by Theorem \ref{Tchordal}. 
\end{proof}

As a consequence of Propositions \ref{Pcoalescence} and \ref{Pcycle}, we
determine the critical exponents of coalescences of cycles. Such graphs,
often called cactus graphs or cactus trees, are useful in applications
and have recently been used to compare sets of related genomes
\cite{paten2010_cactus}. 

\begin{corollary}
Suppose $G$ is a connected cactus graph with at least $3$ vertices.
Then,
\[ \calh_G = \calh_G^\psi = [1,\infty), \qquad [2,\infty) \subset
\calh_G^\phi \subset [1,\infty). \]
\end{corollary}

\begin{proof}
This follows immediately from Propositions \ref{Pcycle} and
\ref{Pcoalescence} and Equation \eqref{Ehcec}.
\end{proof}

\subsection*{Entrywise powers and correlation matrices}

Recall that a correlation matrix is a positive semidefinite matrix with
ones on its main diagonal.
Motivated by applications, a natural question that comes to mind is to
compute the critical exponent for correlation matrices of fixed
dimension, under rank and sparsity constraints.
We now explain why a stronger phenomenon holds: namely, the set of powers
preserving positivity, and hence the critical exponent, remain the same
when restricted to correlation matrices.

\begin{proposition}\label{Pcorrelation}
Fix integers $1 \leq k \leq n$ and a graph $G$ on $n$ vertices. Given $I
\subset \R$, let $\bp_n^k(I)$ denote the set of matrices in $\bp_n(I)$ of
rank at most $k$. Now let $\mathcal{C}_n^k(I)$ and $\mathcal{C}_G(I)$
denote the set of $n \times n$ correlation matrices in $\bp_n^k(I)$ and
in $\bp_G(I)$ respectively. Then, $x^\alpha$ preserves
$\mathcal{C}_n^k([0,\infty))$ (respectively, $\mathcal{C}_G([0,\infty))$)
if and only if $x^\alpha$ preserves $\bp_n^k([0,\infty))$ (respectively
$\bp_G([0,\infty))$). A similar result holds for the power functions
$\phi_\alpha, \psi_\alpha$ acting on $\mathcal{C}_n^k(\R),
\mathcal{C}_G(\R)$.
\end{proposition}

\begin{proof}
Notice that if $A \in \bp_n(\R)$, then setting $D$ to be the diagonal
matrix with entries $\sqrt{a_{jj}}$, we have: $A = DCD$ for some
correlation matrix $C \in \mathcal{C}_n$. Moreover, if $a_{jj} \neq 0\
\forall j$, then $A,C$ have the same rank and sparsity pattern. All
assertions for matrices in $\bp_n((0,\infty))$ now follow by observing
that $A^{\circ \alpha} = D^{\circ \alpha} C^{\circ \alpha} D^{\circ
\alpha}$ for all $\alpha > 0$, whenever $A$ has nonnegative entries. A
similar argument shows all of the assertions for matrices in $\bp_n(\R)$.
\end{proof}

\subsection*{Concluding remarks and questions}

The set of powers preserving positivity was determined for many graphs in
the paper, including chordal graphs, cycles, and complete bipartite
graphs. Apart from computing the $\calh$-sets for every graph, the
following natural questions arise: 
\begin{enumerate}
\item In all of the examples of graphs studied in this paper, it has been
shown that $CE_H(G) = CE_H^\psi(G) = r-2$, where $r$ is the largest
integer such that $G$ contains either $K_r$ or $K_r^{(1)}$ as an induced
subgraph. Does the same result hold for all graphs?

\item Are the critical exponents of a graph always integers? Can this be
shown without computing the critical exponents explicitly?
Do these exponents have connections to other (purely combinatorial) graph
invariants?

\item Recall that every chordal graph is perfect. Can the critical
exponent be calculated for other broad families of graphs such as the
family of perfect graphs?
\end{enumerate}

\subsection*{Acknowledgments}

We would like to thank the American Institute of Mathematics (AIM) for
funding and hosting the workshop ``Positivity, graphical models, and
modeling of complex multivariate dependencies'' in October 2014. The
present paper has benefited from discussions with the participants of the
workshop.
We also thank Shmuel Friedland for useful comments, in particular for
pointing out the result in Proposition \ref{Pcorrelation}. Finally, we
thank the referees for useful remarks and suggestions.



\end{document}